\documentclass[a4paper, 12pt]{amsart}
\usepackage{verbatim}

\usepackage{amsmath,amssymb}
\newtheorem{theorem}{Theorem}[section]

\newtheorem{statement}[theorem]{Statement}

\newtheorem{conjecture}[theorem]{Conjecture}
\newtheorem{corollary}[theorem]{Corollary}

\newtheorem{definition}[theorem]{Definition}

\newtheorem{lemma}[theorem]{Lemma}

\newtheorem{question}[theorem]{Question}

\newtheorem{proposition}[theorem]{Proposition}
\newtheorem{remark}[theorem]{Remark}

\topmargin by -\headsep
\begin{document}

\author{Jonathan Leake and Mohan Ravichandran}
\email{jleake@math.berkeley.edu, mohan.ravichandran@msgsu.edu.tr}
\title{Mixed determinants and the Kadison-Singer problem}


\maketitle

\begin{abstract}
We adapt the arguments of Marcus, Spielman and Srivastava in their proof of the Kadison-Singer problem to prove improved paving estimates. Working with Anderson's paving formulation of Kadison-Singer instead of Weaver's vector balancing version, we show that the machinery of interlacing polynomials due to Marcus, Spielman and Srivastava works in this setting as well. The relevant expected characteristic polynomials turn out to be related to the so called ``mixed determinants'' that have been carefully studied by Borcea and Branden.

This approach allows us to show that any projection with diagonal entries  $1/2$ can be $4$ paved, yielding improvements over the best known current estimates of $12$. This approach also allows us to show that any projection with diagonal entries strictly less than $1/4$ can be two paved, matching recent results of Bownik, Casazza, Marcus and Speegle.

We also relate the problem of finding optimal paving estimates to bounding the root intervals of a natural one parameter deformation of the characteristic polynomial of a matrix that turns out to have several pleasing combinatorial properties.

\keywords{Kadison-Singer problem \and Interlacing polynomials \and Barrier functions \and Real stable polynomials}
\end{abstract}

\section{Introduction}

The Kadison-Singer problem, posed in 1959 \cite{KS59} by Richard Kadison and Isadore Singer, asked if extensions of pure states on the diagonal subalgebra $\ell^{\infty}(\mathbb{N})$ to $\mathcal{B}(\mathcal(\ell^2(\mathbb{N})))$ are unique. This problem was shown to be equivalent to a fundamental combinatorial problem concerning finite matrices by Joel Anderson in 1979 \cite{AndPav}. 

In what follows, we will work with three classes of matrices. A matrix $A \in M_n(\mathbb{C})$ is Hermitian if $A = A^{*}$, where $A^*$ represents the conjugate transpose of $A$. If $A$ has additionally, all eigenvalues non-negative, it is called PSD. A projection matrix $P \in M_n(\mathbb{C})$ is a matrix such that $P = P^{*} = P^2$. Hermitian matrices with all eigenvalues lying in the interval $[-1,1]$ are called contractions.

\begin{question}[Anderson's Paving formulation]\label{Anderson}
Are there universal constants $\epsilon < 1$ and $r \in \mathbb{N}$ so that for any zero diagonal Hermitian matrix $A \in M_n(\mathbb{C})$, there are diagonal projections $Q_1, \cdots, Q_r$ with $Q_1 + \cdots + Q_r  =  I$ such that 
\[||Q_i A Q_i|| < \epsilon\, ||A||, \quad 1 \leq i \leq r?\]
\end{question}

Any such partition of the identity into diagonal projections together with the resulting block compression of the matrix is called a \emph{paving}. Charles Akemann and Joel Anderson gave an alternate formulation of this problem in terms of paving projections in 1991 \cite{AkeAnd}. They showed that a positive solution to it implies a positive solution to the Kadison-Singer problem. 
\begin{question}[Akemann-Anderson's projection formulation]
Are there universal constants $\alpha$ and $\epsilon < 1/2$, so that whenever $P$ is a projection in $M_n(\mathbb{C})$ with diagonal entries at most $\alpha$, there is a diagonal projection $Q$ such that 
\begin{eqnarray}\label{AA}
||QPQ|| < \dfrac{1}{2} + \epsilon \quad \text{and} \quad  || (I-Q)P(I-Q) || < \dfrac{1}{2} + \epsilon?
\end{eqnarray}
\end{question}
 Nik Weaver then gave a interpretation of this in terms of a quantitative vector partitioning problem in 2004 \cite{WeaDis}, and this conjecture of Weaver was solved by Adam Marcus, Dan Spielman and Nikhil Srivastava (MSS) in 2013 \cite{MSS2}, thus resolving the Kadison-Singer problem. They also showed that one may take $\epsilon = \sqrt{2\alpha}+\alpha$ in (\ref{AA}). In the language of the Akemann-Anderson conjecture, MSS proved the following result.
 
  \begin{theorem}[MSS]\label{MSSP}
For any $\alpha > 0$ and any PSD contraction $A \in M_n(\mathbb{C})$ with diagonal entries at most $\alpha$, there are diagonal projections $Q_1 \ldots, Q_r$ such that
\[Q_1 + \ldots + Q_r = I,\]
and
\[||Q_iAQ_i|| \leq \left(\sqrt{\dfrac{1}{r}} + \sqrt{\alpha}\right)^2, \quad i \in [r].\]
\end{theorem}

There are two special classes of paving problems for PSD matrices that are of particular interest. In what follows, if a paving of a matrix has norm strictly less than that of the matrix, we will call it a \emph{non-trivial paving}.

\begin{enumerate}
 \item $2$ paving, namely the case when we pave the PSD matrix into two blocks ($r = 2$). In this setting, the MSS result, Theorem \ref{MSSP}, says that PSD contractions with diagonal entries at most $\alpha = \left(\sqrt{2}-1\right)^2/2 \approx 0.086$ have non-trivial $2$ pavings. This result was improved by Bownik, Casazza, Marcus and Speegle \cite{BCMS}, who showed the same for PSD contractions with diagonal entries all at most $\alpha < 1/4$.
 \item Paving Projection matrices with constant diagonal $1/2$. In this setting, the MSS result, Thm. \ref{MSSP}, says that such matrices have non-trivial $12$ pavings. They deduce from this, using a result of Casazza, Edidin, Kalra and Paulsen \cite{CEKP} that zero diagonal Hermitian matrices, the class of matrices that are the subject of Anderson's paving conjecture, Question \ref{Anderson}, have non-trivial pavings of size $144$. This estimate is suboptimal and finding optimal estimates is a problem of some theoretical interest.
\end{enumerate}

 In the opposite direction, there is a result of Casazza et. al. from \cite{CEKP}, that projections with constant diagonal $1/2$ need not have non-trivial $2$ pavings and it is expected that this is the worst case scenario; that projections with diagonal less than $\alpha < 1/2$ do indeed have non-trivial $2$ pavings. While we do not prove this conjecture in this paper, we prove some weaker paving estimates.  Our main theorem is the following, 
 
 \begin{theorem}
For any integer $r \geq 2$ and $0 < \alpha \leq \left(r-1\right)^2/r^2$ and any PSD contraction $A \in M_n(\mathbb{C})$ with diagonal entries at most $\alpha$, there are diagonal projections $Q_1 \ldots, Q_r$ such that
\[\sum_{i=1}^r Q_i = I, \qquad ||Q_iAQ_i|| \leq \left(\sqrt{\dfrac{1}{r}-\dfrac{\alpha}{r-1}}+ \sqrt{\alpha}\right)^2, \quad i \in [r].\]
\end{theorem}

Note that this quantity is strictly less than $1$ for $\alpha < (r-1)^2/r^2$. 
This matches the estimate of Bownik et. al. from \cite{BCMS}, who consider the case $r = 2$ and also improves the MSS result, Thm. \ref{MSSP}. We remark that Petter Branden, in a recent paper \cite{BraRec}, has also achieved this same result for $r=2$, but his estimates are weaker than the ones we have for general $r$. However his results apply to a much wider class of polynomials than the ones we study in this paper.

As another corollary, we deduce that positive contractions (and thus projections) with diagonal at most $1/2$ have non-trivial $4$ pavings. We remark that non-trivially paving projections with constant diagonal $1/2$ has been the most important quantitative paving problem, ever since the influential survey of Casazza and Tremain \cite{CT}.
\begin{corollary}\label{dhalf}
Let $A \in M_n(\mathbb{C})$ be a PSD contraction with diagonal entries all at most $\alpha \leq 1/2$. Then, there are diagonal projections $\{Q_i : i \in [4]\}$ such that 
\[\sum_{i=1}^4 Q_i = I, \qquad ||Q_i A Q_i||   \leq \dfrac{7+2\sqrt{6}}{12} \approx 0.992, \quad i \in [4].\]
\end{corollary}
Together with a well known result from \cite{CEKP}, this implies that any zero diagonal Hermitian can be $4^2 = 16$ paved with paving constant $\approx 0.984$. This is also unlikely to be optimal, but hopefully our technique can be fine tuned to get optimal results.

As mentioned in the abstract, we approach the Kadison-Singer problem through Anderson's paving formulation \cite{AndPav}, rather than Akemann-Anderson's projection formulation \cite{AkeAnd}, or Weaver's influential vector balancing version \cite{WeaDis}. It turns out that both the major innovations in the work of MSS \cite{MSS2}, the method of interlacing polynomials and the multivariate barrier method, can be directly applied to pavings of matrices. Estimates on the required size of pavings follow from estimates on the locations of roots of certain natural multivariate polynomials, which are closely related to \emph{mixed determinants}\footnote{These are distinct from the more familiar \emph{mixed discriminants} that appear in the work of MSS \cite{MSS2}.}. Interestingly, these are also connected to natural univariate polynomials related to expressions that have appeared in several works independently, called \emph{alpha permanents} \cite{VJ88, VJAP97, Branden2P} by some and  \emph{fermionants} \cite{MerMoo13, ChaWie} by others. 

Let us now briefly outline our approach. In what follows, we will denote the standard basis of $\mathbb{C}^n$ by $\{e_1, \ldots, e_n\}$. Let $A$ be a matrix in $M_n(\mathbb{C})$. For any partition of $[n]$ into $r$ subsets $ \mathcal{X} = \{X_1, \cdots, X_r\}$ (some of the subsets $X_k$ could be empty), we use the notation $A_{\mathcal{X}}$ to denote the corresponding $r$ paving of $A$, 
\[A_{\mathcal{X}} := P_{X_1}AP_{X_1} + P_{X_2}AP_{X_2} + \cdots + P_{X_r}AP_{X_r},\]
where $P_{X_k}$ is the orthogonal projection onto $\operatorname{span}\{e_i : i \in X_k\}$.

There are $r^n$ possible pavings of  $A$ and we use the expression $\mathcal{P}_r$ to denote the set of all $r$ pavings. We will show that when $A$ is Hermitian, the characteristic polynomials of $r$ pavings form an interlacing family in the sense of Marcus, Spielman and Srivastava \cite{MSS2}. As a consequence of their method, one can prove the following,
\begin{theorem}
Let $A\in M_n(\mathbb{C})$ be Hermitian and let $r \in \mathbb{N}$. Then the sum of the characteristic polynomials of all the $r$ pavings is real rooted and further, there is a paving $\mathcal{X} \in \mathcal{P}_r$ such that 
\[\operatorname{ max root }\chi[A_{\mathcal{X}}] \leq \operatorname{ max root } \sum_{\mathcal{X} \in \mathcal{P}_r}  \chi[A_{\mathcal{X}}].\]
\end{theorem}

The expression on the right has several different combinatorial expressions. The first is an expression in terms of differential operators. 

\begin{proposition}\label{rpolymv}
 Let $A \in M_n(\mathbb{C})$ and let $Z$ be the diagonal matrix with diagonal entries $(z_1, \ldots, z_n)$ where the $z_k$ are variables. Then, for any positive integer $r$, 
  \[\sum_{\mathcal{X} \in \mathcal{P}_r}  \chi[A_{\mathcal{X}}](x) =  \left(\dfrac{1}{(r-1)!}\right)^{n}\,\left(\prod_{k=1}^n \dfrac{\partial^{r-1}}{\partial z_k^{r-1}}\right) \operatorname{det}[Z-A]^r\mid_{z_1 = \cdots = z_n = x}.\]
\end{proposition}

The second expression is especially pretty and shows that this expected characteristic polynomial can be written in purely univariate terms.

\begin{definition}
Given a matrix $A = (a_{ij})_{1 \leq i,j \leq n}\in M_n(\mathbb{C})$ and $r \in \mathbb{N}$, define 
\[\operatorname{det}_r [A] := \sum_{\sigma \in S_n}  \prod_{i=1}^{n} a_{i\sigma(i)}\operatorname{sign}(\sigma) r^{c(\sigma)}.\]
where $c(\sigma)$ denotes the number of cycles in $\sigma$. 
\end{definition}
This is the same as the determinant, save for the $r^{c(\sigma)}$ term and in particular, when $r = 1$, specializes to the determinant.  This expression has appeared several times in the mathematical literature and has also shown up in recent work of theoretical computer scientists and physicists \cite{MerMoo13, ChaWie} . A different scaling of this expression, which has been studied in several papers goes under the name of the $\alpha$ permanent \cite{VJ88, VJAP97}. We will write down a polynomial, analogous to the way we define the characteristic polynomial of a matrix. This natural operation has, as far as we know, not been studied so far. 

\begin{definition}\label{chir}
Given a matrix $A \in M_n(\mathbb{C})$, define 
\[\chi_r [A] := \operatorname{det}_r[xI-A].\]
\end{definition}

When $r = 1$, this is the characteristic polynomial. One remarkable feature of this polynomial is that for any positive integer $r$, the polynomial  $\chi_r[A]$ is real rooted for Hermitian $A$. This can be deduced from the interlacing polynomials machinery of MSS, and one can write down several pleasing expressions for this polynomial. For non-integer values of $r$, this polynomial is not real rooted but even in this case, there is an interesting alternate combinatorial expression for $\chi_r[A]$, an expression that is a consequence of McMahon's master theorem, see \cite{FoaZei, VJ88} or \cite{Branden2P}. We will also show that a direct analogue of the Cauchy interlacing theorem holds for any value of $r \in \mathbb{N}$. And most to our point, the expected characteristic polynomial over all pavings turns out to be given by the $r$ characteristic polynomial.

\begin{proposition}
 Let $A \in M_n(\mathbb{C})$ be Hermitian. Then, for any $ r\in \mathbb{N}$,  we have that
 \[ \chi_r[A] = \sum_{\mathcal{X} \in \mathcal{P}_r} \chi[A_{\mathcal{X}}] .\]

\end{proposition}

While we have not been able to find a direct method to estimate the max root of $r$ characteristic polynomials of Hermitian matrices, we feel that analyzing this polynomial is the most promising way of proving optimal paving estimates.

We will obtain root bounds for the sum characteristic polynomial by using the multivariate barrier function method, a general technique to study the evolution of roots of real stable polynomials that was introduced by MSS \cite{MSS2}. A polynomial $p(z_1, \cdots, z_n)$ is said to be real stable if its coefficients are real and it has no zeroes in $\mathbb{H}^n$ where $\mathbb{H} = \{z \in \mathbb{C} : \operatorname{Im}(z) > 0\}$.

 Stability is defined algebraically, but MSS, see also Branden \cite{BraLN} have shown how stable polynomials also enjoy several convexity properties.

Given a real stable polynomial $p \in \mathbb{C}[z_1, \cdots, z_n]$, a point $z = (z_1, \cdots, z_n) \in \mathbb{R}^n$ is said to \emph{above} the roots of $p$ , denoted $z \in \operatorname{Ab}_p$ if $p$ is non-zero in the positive orthant based at $z$, that is
\begin{eqnarray}\label{above}
p(z+t) \neq 0,  \quad \forall t \in \mathbb{R}^n_{+}.
\end{eqnarray}
The Gauss-Lucas theorem implies that the positive orthant based at $z$ is zero free for any partial derivative $\partial_i p$ as well, unless this partial derivative is zero. In particular, if $z \in \operatorname{Ab}_p$, then $z \in \operatorname{Ab}_{\partial_i p}$. However, more is true; taking the partial derivatives of a real stable polynomial with respect to $z_i$ shifts zero free orthants to the left along the direction $e_i$. In other words, one can show that there is a $\delta > 0$ such that $z - \delta e_i \in \operatorname{Ab}_{\partial_i p}$ as well. The multivariate barrier method is a simple but powerful method of getting concrete estimates for how large $\delta$ can be. 

MSS \cite{MSS2} used the multivariate barrier method to get estimates for how zero free orthants evolve under applying operators of the form $1-\partial_i$ in their solution to the Kadison-Singer problem. We apply this method to derivative operators instead. For optimal estimates, we exploit the special structure of the polynomials relevant to Kadison-Singer, not only their degree restrictions as was also done by Bownik et. al. in \cite{BCMS}, but also the fact that they are products of determinantal polynomials.

\section{Characteristic polynomials of pavings}
 Given a matrix $A \in M_n(\mathbb{C})$ and a subset $S \subset [n]$, we use the expression $A_S$ to denote the principal submatrix of $A$ with rows and columns corresponding to elements in $S$ \emph{removed} (this matrix has $n-|S|$ rows and columns).  Also, let $Z$ be the diagonal matrix $Z = \operatorname{diag}(z_1, \cdots, z_n)$, where the $z_k$ are variables. Let us consider the polynomial
  \[\operatorname{det}[Z+A].\]
  This is a multiaffine polynomial in the $z_k$ and it is easy to see that the coefficient of $z^S$ for any subset $S \subset [n]$ equals the determinant of $A_{S}$. Consequently, we have
  \begin{lemma}\label{diff}
  Given $A \in M_n(\mathbb{C})$ and a subset $S \subset [n]$, we have that
 \[\chi[A_{S}](x) = \dfrac{\partial^S}{\partial z^S} \operatorname{det}[Z-A] \mid_{Z = xI}.\]
 \end{lemma}
 
 Central to this paper is the notion of interlacing sequences and polynomials.
 \begin{definition}[Interlacing]
  Two non-increasing real sequences $(\lambda_1, \ldots,\lambda_n)$ and $(\mu_1, \ldots,\mu_n)$ interlace each other if 
  \[\lambda_1 \geq \mu_1 \geq \lambda_2 \geq \mu_2 \geq \ldots \geq \lambda_n \geq \mu_n \quad \text{or} \quad \mu_1 \geq \lambda_1 \geq \mu_2 \geq \lambda_2 \geq \ldots \geq \mu_n \geq \lambda_n.\]
  Similarly, two non-increasing sequences $(\lambda_1, \ldots,\lambda_n)$ and $(\mu_1, \ldots,\mu_{n-1})$ interlace each other if 
  \[\lambda_1 \geq \mu_1 \geq \lambda_2 \geq \mu_2 \geq \ldots \geq \mu_{n-1} \geq \lambda_n .\]
  Finally, two real rooted polynomials $p$ and $q$ interlace each other if they either have the same degree or have degrees differing by one and their roots arranged in non-increasing order interlace each other. 
 \end{definition}

 The celebrated Cauchy-Poincare interlacing theorem says that any defect $1$ principal submatrix\footnote{A defect $k$ principal submatrix of $A \in M_n(\mathbb{C})$ is a $n-k \times n-k$ submatrix obtained from removing $k$ rows and the same $k$ columns} of a Hermitian matrix has the property that its eigenvalues interlace those of the parent matrix. This implies in particular that if $S_1$ and $S_2$ are two equal sized subsets of $[n]$ that differ in exactly one element, then $\chi[A_{S_1}]$ and $\chi[A_{S_2}]$  have a common interlacer, namely $\chi[A_{S_1 \cap S_2}]$. The property of having a common interlacer can be read off from the polynomials at hand, without having to compute a common interlacer, thanks to Obreshkoff's theorem, see \cite{DedOb}, 
 \begin{theorem}[Obreshkoff]\label{Obr}
 Two real rooted univariate polynomials with positive leading coefficient have a common interlacer iff every convex combination of the two is real rooted.
 \end{theorem}
 Also, if two polynomials $p, q$ have a common interlacer, then it is a folklore result \cite{MSS2}, that
 \[\operatorname{min} \{\operatorname{max root} p, \operatorname{max root} q\} \leq \operatorname{max root} (p+q).\]

 MSS \cite{MSS2} introduced the notion of an interlacing family, a gadget that allows one to systematically use eigenvalue interlacing to relate roots of polynomials to roots of their sum.
 
 \begin{definition}[Interlacing Families]\label{InterD}
  
A rooted tree together with monic polynomials associated to each node is called an interlacing family if the following two conditions hold.
 \begin{enumerate}
  \item The polynomial at a (non-leaf) node is the sum of the polynomials associated to its immediate child nodes. 
  \item The polynomials at sibling nodes (nodes with the same parent) all have a common interlacer. 
 \end{enumerate}
 \end{definition}

 Given a Hermitian matrix $A \in M_n(\mathbb{C})$, we now construct an binary tree, that will give us an interlacing family, as follows.
  \begin{definition}[The Matrix Paving Tree]\label{Tree}
 We will consider the following tree.
 \begin{itemize}
 
 \item \textbf{Levels: } This tree will have $n+1$ levels denoted $0$ (the top level) to $n$.
 \item \textbf{Nodes: } The nodes at the bottom or the $n$'th level will correspond to (ordered) partitions of $[n]$ into two subsets\footnote{Through out this paper, an expression of the form $S \amalg T$ will be used when we wish to stress that the sets $S$ and $T$ are disjoint}, that is, $S \amalg T = [n]$. Here, the term `ordered' means that the ordering of the sets will be relevant. For instance, $(\{1, 2\}, \{3\})$ and $(\{3\}, \{1, 2\})$ will be considered to be distinct partitions of $[3]$.
 
 There are $2^n$ such partitions. There will be $2^k$  nodes at level $k$ and they will be indexed by (ordered) partitions of $[k]$ into two subsets. The top node (the single node at level $0$) will be denoted by the empty set $\{\phi\}$. Let us denote the nodes by tuples $(S,T)$, where the level can be read off by finding $k$ such that $S \amalg T = [k]$. 
 \item \textbf{Edges: } Each node save for those at level $n$ (the leaf nodes) will have two children; Given $S, T$ such that $S \amalg T = [k]$, the node $(S,T)$ at level $k$ will have as children $(S\cup \{k+1\}, T)$ and $(S, T \cup \{k+1\})$.
 \item \textbf{Attached Polynomials: } To each node, we will attach a polynomial, which we will denote by $q(S,T)$. Given a node at the bottom level, the polynomial will be 
 \[q(S,T) = \chi[A_S \oplus A_T] = \chi[A_S]\chi[A_T], \quad \text{where} \quad S \amalg T = [n].\] 
 For other nodes, the polynomial will the sum of the polynomials associated to all the leaves under that node.
 \end{itemize}
 \end{definition}
 
 Define $Z = \operatorname{diag}(z_1, \ldots,z_n)$ and $Y = \operatorname{diag}(y_1, \ldots,y_n)$ as diagonal matrices of variables.
 Given $S\amalg T = [k]$, we have
 \begin{eqnarray}\label{MFormula}
 \nonumber q(S,T) &=& \sum_{U \amalg V = [k+1, n]} \chi[A_{S \amalg U} \oplus A_{T \amalg V}] \\
 &=& \nonumber\sum_{U \amalg V = [k+1, n]} \chi[A_{S \amalg U}] \chi[ A_{T \amalg V}],\\
 &\stackrel{\text{Lem. } \ref{diff}}{=}& \nonumber \sum_{U \amalg V = [k+1, n]} \dfrac{\partial^{S \amalg U}}{\partial z^{S \amalg U} } \dfrac{\partial^{T\amalg V}}{\partial y^{T \amalg V}}  \operatorname{det}\left[(Z-A)(Y-A)\right]\mid_{Z = Y = xI} ,\\
 &=&\nonumber \dfrac{\partial^{S}}{\partial z^{S} } \dfrac{\partial^{T}}{\partial y^{T}}\sum_{U \amalg V = [k+1, n]} \dfrac{\partial^{U}}{\partial z^{U} } \dfrac{\partial^{V}}{\partial y^{ V}}  \operatorname{det}\left[(Z-A)(Y-A)\right]\mid_{Z = Y = xI} ,\\
 &=&\nonumber \dfrac{\partial^{S}}{\partial z^{S} } \dfrac{\partial^{T}}{\partial y^{T}}\left(\prod_{m = k+1}^n \dfrac{\partial}{\partial z_m}+\dfrac{\partial}{\partial y_m} \right) \operatorname{det}\left[(Z-A)(Y-A)\right]\mid_{Z = Y = xI}. 
 \end{eqnarray}
 
 In particular, the top node is
 
  \begin{eqnarray}\label{MFormula2}
  q(\{\phi\}) = \left(\prod_{m = 1}^n \dfrac{\partial}{\partial z_m}+\dfrac{\partial}{\partial y_m} \right)\operatorname{det}\left[(Z-A)(Y-A)\right]\mid_{Z = Y = xI} 
 \end{eqnarray}

 We now show that this family is an interlacing family in the sense of MSS.
Our proof will use basic algebraic properties of real stable polynomials, akin to \cite{MSS2}. We recall the definition of real stable polynomials.
\begin{definition}[Stable and Real Stable polynomials]
 A polynomial $p \in \mathbb{C}[z_1, \cdots, z_n]$ is called stable if it is non-vanishing on $\mathbb{H}^{n}$, where $\mathbb{H}$ is the open upper half plane, $\mathbb{H} = \{ z \in \mathbb{C} : \operatorname{Im}(z) > 0\}$. A stable polynomial with real coefficients is called real stable.
\end{definition}
One basic class of real stable polynomials come from multivariate characteristic polynomials. 
\begin{lemma}\label{detrs}
 Let $A \in M_n(\mathbb{C})$ be Hermitian and let $Z = \operatorname{diag}(z_1, \cdots, z_n)$. Then, the polynomial 
 \[p(z_1, \ldots, z_n) := \operatorname{det}[Z-A],\]
 is real stable. 
\end{lemma}
\begin{proof}
 If $\operatorname{det}[Z-A] = 0$, then, there is a non-zero vector $v \in \mathbb{C}$ such that $v^{*}(Z-A) v = 0$.  Let $\operatorname{Im}(z_i) > 0, \, i \in [n]$. We have, 
 \[
\operatorname{Im}\, v^{*}(Z-A) v = \operatorname{Im}\,\left(v^{*}Z v - v^{*}A v\right) = v^{*}\left(\operatorname{Im} Z\right)v > 0,
 \]
 yielding the desired contradiction. 

\end{proof}

The following basic properties of stable and real stable polynomials are well known \cite{BBLY2}, and can be easily verified. 
\begin{proposition}\label{rsprop}
 Let $p \in \mathbb{C}[z_1, \cdots, z_n]$ be stable. Then, the following are also stable unless they are identically zero, 
 \begin{enumerate}
  \item Given non-negative reals $(a_1, \cdots, a_n)$, the polynomial $(\sum_{i \in n} \alpha_i \partial_i)\,p$.
  \item Given $a \in \mathbb{C}$ with $\operatorname{Im}(a) \geq 0$, the $n-1$ variate polynomial $p(a,z_2, \cdots, z_n)$. 
 \end{enumerate}
 If $p$ is additionally real stable (i.e. it has real coefficients), then the following are also real stable unless they are identically zero, 
 \begin{enumerate}
  \item Given non-negative reals $(a_1, \cdots, a_n)$, the polynomial $(\sum_{i \in n} \alpha_i \partial_i)\,p$.
  \item Given $a \in \mathbb{R}$, the $n-1$ variate polynomial $p(a,z_2, \cdots, z_n)$. 
 \end{enumerate}

\end{proposition}
With this in hand, we can prove our desired interlacing family result. 

 \begin{lemma}
  The Matrix Paving Tree from Definition \ref{Tree} yields an interlacing family. 
 \end{lemma}
 \begin{proof}
 The first condition in the definition of interlacing families (see Definition \ref{InterD}) holds by construction. We now show that any two sibling nodes have a common interlacer. 
 
  Let $A$ and $B$ be two sibling nodes at level $k$ where $k \in \{1, 2, \ldots , n\}$. Then there is a partition $S \amalg T = [k-1]$ such that the polynomials associated to $A$ and $B$ are respectively, $q(S \cup \{k\}, T)$ and $q(S, T \cup \{k\})$. By Obreshkoff's theorem (see Theorem \ref{Obr}), we need to show that for every $0 \leq \alpha \leq 1$, we have that
  \[q := \alpha \,q(S \cup \{k\}, T) + (1-\alpha) \,q(S, T \cup \{k\}),\] is real rooted. Let $p$ be the polynomial
  \[p(Z,Y) = p(z_1, \ldots, z_n, y_1, \ldots, y_n) := \operatorname{det}[(Z-A)(Y-A)],\]
  where as previously, $Z = \operatorname{diag}(z_1,\ldots,z_n)$ and $Y = \operatorname{diag}(y_1,\ldots,y_n)$ are $n \times n$ diagonal matrices of variables. We have by Formula \ref{MFormula}, 
  \begin{eqnarray*}
 q &:=& \alpha \,q(S \cup \{k\}, T) + (1-\alpha) \,q(S, T \cup \{k\}) \\
 &=&
  \left[\alpha \, \dfrac{\partial}{\partial z_{k}}\dfrac{\partial^{S}}{\partial z^{S} } \dfrac{\partial^{T}}{\partial y^{T}} + (1-\alpha)\dfrac{\partial}{\partial y_{k}}\dfrac{\partial^{S}}{\partial z^{S} } \dfrac{\partial^{T}}{\partial y^{T}}  \right] \left(\prod_{m = k+1}^n \dfrac{\partial}{\partial z_m}+\dfrac{\partial}{\partial y_m}\right) p(Z,Y)\mid_{Z = Y = xI} \\
 &=&  \left(\alpha \dfrac{\partial}{\partial z_k} + (1-\alpha)\dfrac{\partial}{\partial y_k} \right) \dfrac{\partial^{S}}{\partial z^{S} } \dfrac{\partial^{T}}{\partial y^{T}}  \left(\prod_{m = k+1}^n \dfrac{\partial}{\partial z_m}+\dfrac{\partial}{\partial y_m} \right)p(Z,Y)\mid_{Z = Y = xI} .
 \end{eqnarray*}
 Now, note that $p$ is real stable by Lemma \ref{detrs}, and  partial derivatives as well as non-negative linear combinations of partial derivatives preserve real stability and further, specializing variables to real scalars preserves real stability by Proposition \ref{rsprop}. We conclude that the polynomial $q$ is real rooted. 
 \end{proof}  
 
MSS \cite{MSS2}[Theorem 3.4] show that given an interlacing family, there is at least one leaf node whose polynomial has max root less than or equal to the max root of the polynomial at the root node. This immediately implies, 

  \begin{theorem}\label{expcom}
  Let $A\in M_n(\mathbb{C})$ be Hermitian. Then, the sum of the characteristic polynomials of all the $2$ pavings of $A$ is real rooted and satisfies 
  \[\sum_{S \amalg T = [n]} \chi[A_S \oplus A_T] = \left[\prod_{m = 1}^n \dfrac{\partial}{\partial z_m}+\dfrac{\partial}{\partial y_m} \right] p(Z,Y)\mid_{Z = Y = xI} .\]
  Further, there is a paving $(S,T) \in \mathcal{P}_2([n])$ such that 
  \[\operatorname{max root} \chi[A_S \oplus A_T] \leq \operatorname{max root}\sum_{S \amalg T = [n]} \chi[A_S \oplus A_T].\]

  \end{theorem}
  
  This analysis can be carried out for $r$ pavings as well for any $r \in \mathbb{N}$. The proof is similar and we omit it. In the following theorem, for $k \in [r]$, we let $Z_k$ be the diagonal matrix with entries $(z_{k1}, \cdots, z_{kn})$.
  
  \begin{theorem}\label{part1}
  Let $A\in M_n(\mathbb{C})$ be Hermitian. Then, the sum of the characteristic polynomials of all the $r$ pavings of $A$ is real rooted and satisfies
  \[ \sum_{\mathcal{X} \in \mathcal{P}_r([n])} \chi[A_{\mathcal{X}}] = \left[\prod_{m = 1}^n \left(   \sum_{k = 1}^{r} \dfrac{\partial}{\partial z_{km}} \right)^{r-1} \prod_{k=1}^{r}\operatorname{det}[Z_k-A]^r\right]\mid_{Z_1= \cdots = Z_r = xI} .\]
  Further, there is a paving $\mathcal{X}\in \mathcal{P}_r([n])$ such that 
  \[\operatorname{max root}\chi[A_{\mathcal{X}}] \leq \operatorname{max root}\sum_{\mathcal{X} \in \mathcal{P}_r([n])} \chi[A_{\mathcal{X}}]   .\]

  \end{theorem}
  
  In the next section, we will derive other useful expressions for this expected characteristic polynomial. 

 We now show how this fact, that one can use expected characteristic polynomials to get estimates about one paving can be understood in a more general framework. The concept of a \emph{Strongly Rayleigh} measure was introduced by Borcea, Branden and Liggett \cite{BBL09}, in order to develop a systematic theory of negative dependence in probability. The main MSS theorem was extended to the setting of Strongly Rayleigh measures by Anari and Oveis Gharan in \cite{AGKS}, and we would like to point out how a version of the algebraic component of their results holds in our setting. 
 
 Recall that a probability distribution $\mu$ on $\mathcal{P}([n])$ is said to be \emph{Strongly Rayleigh} if the generating polynomial, 
 \[P_{\mu} = \sum_{S \subset [n]} \mu(S)z^{S},\]
 is real stable. An adaptation of the proof of Theorem \ref{expcom} shows the following, 
  \begin{theorem}\label{SR}
 Let $\mu$ be a Strongly Rayleigh distribution on $\mathcal{P}([n])$ and let $A \in M_n(\mathbb{C})$ be Hermitian. Then
 \[\mathbb{E}_{S \sim \mu} \chi[A_S] := \sum_{S \subset [n]} \mu(S) \chi[A_S], \]
 is real rooted and, 
  \[\mathbb{P}_{S \sim \mu} \left[ \operatorname{max root} \chi[A_S] \leq \operatorname{max root} \mathbb{E} \chi[A_S] \right]  > 0.\] 
  Further, we have the following formula for the expected characteristic polynomial, 
   \[\mathbb{E} \chi[A_S](x) = P_{\mu}(\partial_1, \cdots, \partial_n)\operatorname{det}[Z-A] \mid_{Z = xI}.\]
  \end{theorem}
  
  This theorem can be specialized to cover two cases of interest. The first is the restricted invertibility problem that in one incarnation asks for the following.
  
  \begin{question}[The restricted invertibility problem]
  Given a PSD matrix $ A\in M_n(\mathbb{C})$, find a subset $S \subset [n]$ of size $k$ such that the principal submatrix $A_S$ has norm at most $\epsilon$.
  \end{question}
  
  For a given $\epsilon$, one would like $k$ to be as small as possible and conversely, for a given $k$, one would like $\epsilon$ to be as small as possible. The best current estimate due to Marcus, Spielman, and Srivastava uses the method of Interlacing polynomials \cite{MSS3}. A version of their approach is as follows : We will apply Theorem \ref{SR} to the  uniform measure $\mu$ over all $n-k$ element subsets, whose generating polynomial is
  \[P_\mu = \binom{n}{k}^{-1} \sum_{|S|  = n-k} z^S =   \binom{n}{k}^{-1} (\partial_1+\cdots + \partial_n)^{k}z_1\cdots z_n.\]
  This shows that restricted invertibility bounds can be derived from estimating the max root of an expected characteristic polynomial, which can be easily shown to be just the $k$'th derivative of the characteristic polynomial of $A$.
  
  For Kadison-Singer, let's look at  two paving first. Let $A \in M_n(\mathbb{C})$ be a Hermitian matrix. Consider the set $[2n]$, which we write as $[n] \amalg  [n]$. We now choose the measure $\mu_2$ that is uniform on subsets of the form $(S, [n]\setminus S)$, where $S \subset [n]$. More generally, when it comes to $r$ paving, we look at $\overbrace{[n] \amalg \cdots \amalg [n]}^{r} \sim [rn]$ and consider the uniform measure, which we denote $\mu_r$ on $S_1 \times \cdots \times S_r$ where $S_1^c \amalg \cdots \amalg S_r^c = [n]$.

Using variables $(z_{i1}, \cdots, z_{in})$ to represent the atoms in the $i'th$ copy of $[n]$, we have that the generating polynomial is 
  \[P_{\mu_r} = r^{-n}\prod_{m = 1}^{n}\left(\dfrac{\partial}{\partial z_{1m}} + \cdots + \dfrac{\partial}{\partial z_{rm}}\right) \left(\prod_{i = 1}^{r}\prod_{j = 1}^{n}z_{ij}\right).\]
  
  We now apply Theorem \ref{SR} to the $rn \times rn$ matrix $\overbrace{A \oplus \cdots \oplus A}^r$. It is easy to see that it gives us that there is a $r$ paving $\mathcal{X} = X_1 \amalg \cdots \amalg X_r$ such that 
  \[\operatorname{max root} A_{\mathcal{X}} \leq \operatorname{max root}\,\mathbb{E}_{S \sim \mu}\chi[A_S].\]

  In the next section, we will show that the above expected characteristic polynomials---the ones that are relevant for Kadison-Singer---have other, even more pleasant expressions. 
  
  \section{Expected characteristic polynomials}
  
  The expression for the sum of the characteristic polynomials over all $r$ pavings in Theorem \ref{part1} will allow us to prove strong estimates on its roots, but might seem unwieldy. In this section, we give two other expressions for this polynomial, and we discuss some of their more interesting and important features.
  
  It should be noted that, except for the interlacing proposition, Proposition \ref{rcharint}, most of what follows will not be explicitly used in the rest of this paper. However, we have chosen to include this discussion, as the $r$ characteristic polynomial (see Definition \ref{rcharint}) seems to be the most natural object for studying analytic properties of pavings. The results which follow are also quite natural, in the sense that they emulate properties of the regular characteristic polynomial. We hope that this extra exposition will then make the case for the $r$ characteristic polynomial being an interesting object of study. 
  
  Our first observation is that we may write the expected characteristic polynomial over all pavings out in a simpler way, that is reminiscent of the mixed characteristic polynomial of MSS.
  
\begin{lemma}\label{diffexp}
Let $A \in M_n(\mathbb{C})$. Then, for any positive integer $r$, we have that
  
   \[\sum_{\mathcal{X} \in \mathcal{P}_r([n])} \chi[A_{\mathcal{X}}] = \left(\dfrac{1}{(r-1)!}\right)^n \dfrac{\partial^{(r-1)n}}{\partial z_1^{r-1}\cdots \partial z_n^{r-1}} \operatorname{det}[Z-A]^r\mid_{z_1 = \cdots = z_n = x}.\]

\end{lemma}
\begin{proof}
By the product (Leibniz) rule, we have
\[ \dfrac{\partial^{(r-1)n}}{\partial z_1^{r-1}\cdots \partial z_n^{r-1}} \operatorname{det}[Z-A]^r = [(r-1)!]^n\sum_{(S_1,\ldots,S_r)\in \mathcal{Y}} \prod_{k=1}^r  \left(\dfrac{\partial^{S_k}}{\partial z^{S_k}}\operatorname{det}[Z-A]\right)\]
where the set $\mathcal{Y}$ that we take summation over is defined as the collection of tuples $(S_1, \ldots, S_r)$ such that  
\begin{enumerate}
\item Each $S_k$ is a multiset containing elements from $[n]$,
\item Each element from $[n]$ occurs exactly $r-1$ times in $\cup_{k=1}^r S_k$,

So in particular, each element from $[n]$ occurs at most $r-1$ times in each $S_k$ for $k \in [r]$. Further, since 
\[\dfrac{\partial^{S_k}}{\partial z^{S_k}}\operatorname{det}[Z-A],\]
is zero if $S_k$ contains any element from $[n]$ more than once, we can actually reduce to
\item Each element from $[n]$ occurs at most once in each $S_k$ for $k \in [r]$.
\end{enumerate}

Together, these imply that $(S_1^c, \ldots, S_r^c)$ form a $r$ partition of $[n]$. This is because condition $(3)$ implies that the $S_k$ are actually sets, not multisets, and hence $S_k^c$ makes sense as a set. Condition $(2)$ then implies that this is a $r$ partition. 

We see that
\begin{align*}
\dfrac{1}{[(r-1)!]^n}\dfrac{\partial^{(r-1)n}}{\partial z_1^{r-1}\cdots \partial z_n^{r-1}} \operatorname{det}[Z-A]^r &= \sum_{(S_1,\ldots,S_r)\in \mathcal{Y}} \prod_{k=1}^r  \left(\dfrac{\partial^{S_k}}{\partial z^{S_k}}\operatorname{det}[Z-A]\right)\\
&= \sum_{(S_1,\ldots,S_r)\in \mathcal{Y}} \prod_{k=1}^r  \operatorname{det}[(Z-A)_{S_k}]\\
&= \sum_{S_1^c \amalg \ldots \amalg S_r^c = [n]}   \operatorname{det}[\sum_{k=1}^{r}P_{S_k^c}(Z-A)P_{S_k^c}].
\end{align*}
This last expression upon specializing to $z_1 = \ldots = z_n = x$ is precisely the sum of all the characteristic polynomials of $r$ pavings. The lemma follows.
 \end{proof}

Let us derive another expression for this polynomial. Consider the expression,
\[ \left(\dfrac{1}{(r-1)!}\right)^n\dfrac{\partial^{(r-1)n}}{\partial z_1^{r-1}\cdots \partial z_n^{r-1}} \operatorname{det}[Z+A]^r\mid_{z_1 = \cdots = z_n = 0}\]
This is the coefficient of $(z_1\cdots z_r)^{r-1}$ in the polynomial $\operatorname{det}[Z+A]^r$. We may expand out $\operatorname{det}[Z+A]$ as 
\[ \operatorname{det}[Z+A] = \sum_{S \in [n]}  z^{S} \operatorname{det}[A_S].\]
Expanding out $\operatorname{det}[A_S]$, we have

\begin{eqnarray*}
 \operatorname{det}[Z+A] &=& \sum_{S \in [n]} \sum_{\sigma \in \operatorname{Aut}([n]\setminus S)}  z^{S} \operatorname{sign}(\sigma)\prod_{i  \in [n]\setminus S} a_{i\sigma(i)}
 \end{eqnarray*}
 Consequently, for any $r \in \mathbb{N}$,
 \begin{eqnarray*}
 \operatorname{det}[Z+A]^r &=& \sum_{S_1, \ldots, S_r \in [n]} \prod_{k=1}^{r} z^{S_k} \sum_{\substack{\sigma_k \in \operatorname{Aut}([n]\setminus S_k), \\ k \in [r]}}  \prod_{k =1 }^{r}\left( \operatorname{sign}(\sigma_k)\prod_{i  \in [n]\setminus S_k} a_{i\sigma_k(i)}\right)
 \end{eqnarray*}
 The terms that contribute to the coefficient of $(z_1\ldots z_n)^{r-1}$ are those where each element in $[n]$ occurs in $r-1$ of the sets $S_1, \ldots, S_r$. Alternately, $S_1^c, \ldots, S_r^c$ must form a partition of $[n]$. In this case, the expression 
 \[ \prod_{k =1 }^{r}\left( \operatorname{sign}(\sigma_k)\prod_{i  \in [n]\setminus S_k} a_{i\sigma_k(i)}\right),\]
 can be written as
 \[ \operatorname{sign}(\sigma) \prod_{i=1}^{n} a_{i\sigma(i)},\]
 where $\sigma$ is the permutation that restricts to the sub permutations $\sigma_k$ on each $S_k^{c}$ for $k \in [r]$.
 
 The number of times this last permutation arises is precisely equal to the number of ways it can be written as such a product of $r$ sub-permutations on subsets, where the ordering is taken into account. This in turn amounts to assigning each cycle in the permutation to one of the $r$ expressions we take the product of, yielding that the term arises precisely $r^{c(\sigma)}$ times. As a consequence, from Definition \ref{chir}, we have that
 \[\operatorname{det}_r[A] = \left(\dfrac{1}{(r-1)!}\right)^n\dfrac{\partial^{(r-1)n}}{\partial z_1^{r-1}\cdots \partial z_n^{r-1}} \operatorname{det}[Z+A]^r\mid_{z_1 = \cdots = z_n = 0}.\]
 This in turn implies that
 \[\chi_r[A] = \left(\dfrac{1}{(r-1)!}\right)^n\dfrac{\partial^{(r-1)n}}{\partial z_1^{r-1}\cdots \partial z_n^{r-1}} \operatorname{det}[Z+A]^r\mid_{z_1 = \cdots = z_n = x}.\]
 Taken together with Lemma \ref{diffexp}, this implies the following,
\begin{lemma}\label{chiexp}
Let $A \in M_n(\mathbb{C})$. Then, we have that,  
 \[\sum_{\mathcal{X} \in P_r([n])}\chi[A_{\mathcal{X}}]   = \chi_r[A].\]
\end{lemma}

As discussed above, the $r$ characteristic polynomial (for positive integer values of $r$) shares some of the features of the regular characteristic polynomial (the case when $r = 1$). We have already seen that the roots of $\chi_r[A]$ for any Hermitian matrix $A$ are real (just combine Theorem \ref{part1} and Lemma \ref{chiexp}). The propositions that follow will demonstrate a few more of these similar features, with the proofs mainly relying on the following two basic identities. Here, $A_i$ denotes the principal submatrix of $A$ with the $i'th$ row and column removed.
\begin{eqnarray}
    \det[Z_i-A_i]^r &=& \frac{1}{r!} \frac{\partial^r}{\partial z_i^r} \det[Z-A]^r,\\
    {\det}_r[Z_i-A_i] &=& \frac{1}{r} \frac{\partial}{\partial z_i} {\det}_r[Z-A].
    \end{eqnarray}

We first prove the most important of these similar features: that the $r$ characteristic polynomial enjoys the same interlacing properties that the characteristic polynomial does.

\begin{proposition}[Cauchy Interlacing]\label{rcharint}
Let $A \in M_n(\mathbb{C})$ be Hermitian and let $r \in \mathbb{N}$. For any $i \in [n]$ we have that the roots of $\chi_r[A]$ and $\chi_r[A_i]$ interlace. 
\end{proposition}
\begin{proof}
    Note the following, where the second follows from the above identities:
    \[
        \chi_r[A](x) = \left.{\det}_r[Z-A]\right|_{z_1=...=z_n=x}
    \]
    \[
        \chi_r[A_i](x) = \left.\frac{1}{r} \frac{\partial}{\partial z_i}{\det}_r[Z-A]\right|_{z_1=...=z_n=x}
    \]
    It is readily checked that $a + \frac{b}{r} \frac{\partial}{\partial z_i}$ preserves real stability for all $a,b \in \mathbb{R}$. Consequently,
    \[
        a \chi_r[A](x) + b \chi_r[A_i](x) = \left.\left(a + \frac{b}{r} \frac{\partial}{\partial z_i}\right) {\det}_r[Z-A]\right|_{z_1=...=z_n=x}
    \]
    is real-rooted for all $a,b \in \mathbb{R}$, and the result follows from Obreshkoff's theorem.
\end{proof}

Another feature the $r$ characteristic polynomial shares with the regular characteristic polynomial is an analogue of Thompson's formula, see \cite{ThoPS1}, that the sum of characteristic polynomials of defect $1$ submatrices equals the derivative of the characteristic polynomial of the original matrix. 

 \begin{proposition}[Thompson type formula]
Let $A \in M_n(\mathbb{C})$ be a not necessarily Hermitian matrix and let $r \in \mathbb{N}$. Then,
\[r \sum_{i \in [n]} \chi_r[A_i] = \chi_r'[A].\]
\end{proposition}
\begin{proof}
    Note the following algebraic identity for any polynomial $p(z_1,...,z_n)$:
    \[
        \sum_{i \in [n]} \left.\frac{\partial}{\partial z_i}p \right|_{z_1=...=z_n=x} = \frac{\partial}{\partial x}\left(p(x,...,x)\right)
    \]
    Using the above identites, this implies:
    \[
        r \sum_{i \in [n]} \chi_r[A_i] = \sum_{i \in [n]} \left. \frac{\partial}{\partial z_i} {\det}_r[Z-A] \right|_{z_1=...=z_n=x} = \frac{\partial}{\partial x} \chi_r[A]
    \]
\end{proof}

A simple induction argument, then shows that sums of $r$ characteristic polynomials of defect $k$ principal submatrices can be related to the $k$'th derivative of $\chi_r[A]$.
\begin{corollary}
Let $A \in M_n(\mathbb{C})$ be a not necessarily Hermitian matrix and let $r \in \mathbb{N}$. Then, for any $k \in [n]$, 
\[r^k k! \sum_{S \subset [n], |S| = k} \chi_r[A_S] = \chi_r^{(k)}[A].\]
\end{corollary}


We finally record an interesting identity for a multiaffine version of the $r$ characteristic polynomial. 

 \begin{proposition}[Multilinearization]
Let $A \in M_n(\mathbb{C})$ be a not necessarily Hermitian matrix and let $r \in \mathbb{N}$. Then, letting $Z$ as usual be a diagonal matrix of variables, $Z = \operatorname{diag}(z_1, \cdots, z_n)$, we have,
\[\operatorname{det}_r[Z-A] = \left(\dfrac{1}{(r-1)!}\right)^n \dfrac{\partial^{(r-1)n}}{\partial z_1^{r-1}\cdots \partial z_n^{r-1}} \operatorname{det}[Z-A]^r.\]
\end{proposition}
\begin{proof}
    We prove this statement by induction on $n$. When $n=1$, the statement is trivial. For larger $n$, we show that all partial derivatives, as well as the constant term, of the two sides are equal. For the partial derivatives, we use the above identities and induct:
    \[
        \begin{split}
            \frac{\partial}{\partial z_i} {\det}_r[Z-A] &= r \cdot {\det}_r[Z_i-A_i] \\
                &= r \cdot \frac{1}{r!} \frac{\partial}{\partial z_i} \left(\frac{1}{(r-1)!}\right)^{n-1} \frac{\partial^{(r-1)n}}{\partial z_1^{r-1} \cdots \partial z_n^{r-1}} \det[Z-A]^r
        \end{split}
    \]
    The last thing to prove then is that the constant terms of the two sides are equal. Notice that this follows from plugging in $x=0$ in Lemmas \ref{diffexp} and \ref{chiexp}.
\end{proof}


We summarize this in a separate corollary, 
 \begin{corollary}
Let $A \in M_n(\mathbb{C})$ be a not necessarily Hermitian matrix and let $r \in \mathbb{N}$. Then, 
\[\operatorname{det}_r[Z+A] = \sum_{S \subset [n]} z^S r^{|S|} \operatorname{det}_r[A_S] \]
\end{corollary}

We now use the observation that $\operatorname{det}_r(Z+A)$ is a multiaffine real stable polynomial to conclude that the measure $\mu$ defined on $\mathcal{P}([n])$ by 
\[\mu(S) = r^{|S|} \operatorname{det}_r[A_S],\]
is a \emph{Strongly Rayleigh} measure (see \cite{BBL09}, where these were introduced, for the definition and a discussion). This immediately implies an analogue of the Hadamard-Fischer-Koteljanski inequalities, 
\begin{proposition}
Let $A \in M_n(\mathbb{C})$ be PSD and let $r \in \mathbb{N}$. Then, for any two subsets $S, T \subset [n]$, we have that
\[\operatorname{det}_r[A_S]\operatorname{det}_r[A_T] \geq \operatorname{det}_r[A_{S \cap T}]\operatorname{det}_r[A_{S\cup T}].\]
\end{proposition}

The $r$ characteristic polynomial is closely related to the \emph{mixed determinant} (not to be confused with the \emph{mixed discriminant}), that is defined for tuples of matrices \cite{BBJ}. Given $n \times n$ matrices $A_1, \cdots, A_k$, the mixed determinant of the tuple is defined as
\[D(A_1, \cdots, A_k) := \sum_{S_1 \amalg \cdots \amalg S_k = [n] }\operatorname{det}[A(S_1)] \cdots \operatorname{det}[A(S_k)],\]
where $A(S)$ denotes the principal submatrix formed by selecting the rows and columns from $S$. 
It is immediate that
\[\chi_r[A](x) = D(\overbrace{xI-A, \cdots, xI-A}^r).\]
Borcea and Branden \cite{BBJ} proved a variety of interlacing results for polynomials of the form $D(xA,B)$, which generalize the regular characteristic polynomial, because of the identity, $\chi[A](x) = D(xI, -A)$.

There is another expression for the $r$ determinant, a consequence of MacMahon's Master theorem, see \cite{FoaZei, VJ88}  or \cite{Branden2P}, which works for non integral values of $r$ as well.
\begin{theorem}
 Let $A \in M_n(\mathbb{C})$ and let $Z$ be the diagonal matrix with diagonal entries $(z_1, \cdots, z_n)$ where the $z_k$ are variables. Then, for any $r\in \mathbb{R}$, we have that
 \[\operatorname{det}_r[A] = \dfrac{\partial^{n}}{\partial z_1\cdots \partial z_n} \operatorname{det}[I-ZA]^r\mid_{z_1 = \cdots = z_n = 0}.\]
Consequently, 
  \[\chi_r[A](x) = \dfrac{\partial^{n}}{\partial z_1\cdots \partial z_n} \operatorname{det}[I-xZ+ZA]^r\mid_{z_1 = \cdots = z_n = 0}.\]
  \end{theorem}

  \begin{remark}
For non-integer values of $r$, this polynomial is not real rooted. For instance, let $J_4$ be the $4 \times 4$ matrix with
\[ J_4(i, j) = 1, \quad  i, j \in [4].\] Then, it is easy to check that $\chi_r[J_4]$ is not real rooted for $r \in (1,2)$. It is also possible to show using matrices of the form $J_k$ that the only values of $r$ such that $\chi_r[A]$ is real rooted for every Hermitian matrix, are the positive integers. 
\end{remark}

In the next section, we prove estimates on the maximum roots of $\chi_r$ and discuss plausible estimates for general $r$. We feel the following is true, 
\begin{conjecture}\label{conj}
Let $r \in \mathbb{N}\setminus\{1\}$ and let $A \in M_n(\mathbb{C})^{+}$ be a positive contraction and let the diagonal entries of $A$ all be at most $\alpha$. Then, for any $r \geq 1/(1-\alpha)$,

\[\operatorname{max root} \chi_r[A] \leq \dfrac{1}{r} \left(\sqrt{1-\alpha} + \sqrt{(r-1)\alpha}\right)^2, \quad i \in [r].\]

\end{conjecture}
While we are unable to prove this, we do prove a weaker result that is also asymptotically optimal. 

\section{The Multivariate Barrier method}

In this section, we prove bounds for the largest root of the $r$ characteristic polynomial, Theorem \ref{Main}. Let $A \in M_n(\mathbb{C})$ be a Hermitian matrix and recall that the $r$ characteristic polynomial is equal to
\[\chi_r[A](x) = \dfrac{\partial^{(r-1)n}}{\partial z_1^{r-1} \cdots \partial z_n^{r-1}} \operatorname{det}[Z-A]^r \mid_{Z = xI}.\]

Let $p$ be the polynomial,
\[p(z) := \operatorname{det}[Z-A]^r .\]
For any subset $S \subset [n]$, we define
\[p_S(z) := \left(\prod_{i \in S} \dfrac{\partial^{r-1}}{\partial z_i^{r-1}}\right)p.\]

Each of these polynomials $p_S$ is of degree  $r$ in $z_i$ for $i \in S^c$ and of degree one in the $z_i$ for $i \in S$. All of these polynomials are further, real stable. 

For a real stable polynomial $q(z_1,\ldots,z_n)$, define the barrier function in the direction $i$ at a point $z$ that is above the roots of $q$ (denoted $z \in \operatorname{Ab}_q$), see Definition \ref{above}  by
\[\Phi_q^i(z) := \dfrac{\partial_i q}{q}(z).\]
As pointed out by MSS \cite{MSS2}, at any point $z \in \operatorname{Ab}_q$, we have
\[\Phi_q^i(z) \geq 0, \quad \partial_j \Phi_q^i(z) \leq 0, \quad \partial_j^2 \Phi_q^i(z) \geq 0, \quad i,j \in [n].\]
We further have that
\begin{eqnarray}\label{phiform}
\Phi_{\partial_i q}^{j} = \dfrac{\partial_j \partial_i q}{\partial_i q} = \dfrac{\partial_j \left(q \Phi^i_q\right)}{q \Phi^i_q} =  \Phi_{q}^{j} + \dfrac{\partial_j \Phi_q^i}{\Phi_q^i}.
\end{eqnarray}
Since $\partial_j \Phi^i_q \leq 0$, this shows that $\Phi_{\partial_i q}^{j} \leq \Phi_{q}^{j}$ and the essence of the barrier method is to get estimates on the largest $\delta$ such that 
\begin{eqnarray}\label{barr}
\Phi_{\partial_i q}^{j}(z - \delta e_i) \leq \Phi_{q}^{j}(z).
\end{eqnarray}

In our proof, we will need a stronger statement than the monotonicity and convexity of the barrier functions.

\begin{proposition}\label{tao}
 Let $p(z_1,z_2)$ be a real stable polynomial of degree $r$ in $z_2$. For each $a \in \mathbb{R}$, the polynomial $z \rightarrow p(a,z)$ is univariate and real stable and thus real rooted. Denoting its roots by $\lambda_1(a) \geq \ldots \geq \lambda_r(a)$, we have that for any $k \in [r]$, the map
 \[a \rightarrow \lambda_k(a),\]
 defined on $\mathbb{R}$ is non increasing. 
 \end{proposition}

This fact is well known and a proof was given by Terence Tao in his expository post on the MSS solution to the Kadison-Singer problem, but we include a proof for completeness.

\begin{proof}
Since $p$ is a polynomial, the functions $\lambda_k(z)$ are locally analytic around $a$, except possibly when there is a multiple root ($\lambda_k(a) = \lambda_i(a)$ for some $i \neq k$). Further, it well known that the roots of a polynomial vary continuously with the coefficients. Combining these two observations, it suffices to prove this when the map $\lambda_k(\cdot)$ is locally differentiable around $a$ (and thus locally analytic).

If $\partial \lambda_k(a) > 0$, then for small positive $\delta$, $b = \lambda_k(a + i\delta)$ would have positive imaginary part and since $p(a + i\delta,b) = 0$, this would contradict the real stability of $p$. 
\end{proof}

What follows is our basic result on how taking derivatives affects the log barrier. 
\begin{proposition}
 Let $p(\textbf{z}) = p(z_1, \ldots, z_n)$ be a real stable polynomial of degree at most $r$ in $z_i$ and let $\textbf{a} = (a_1, \ldots, a_n) \in Ab_p$. Then, for any $j \in [n]$,
 \[\Phi_{\partial_j^{r-1} p}^{i}(\textbf{a} - \delta e_j) \leq \Phi_p^i(\textbf{a}),\]
 provided
 \[\partial_i\left(\dfrac{\partial_j^{r-1}p}{p}\right)(\textbf{a})  \leq \delta\, \partial_i\left(\dfrac{\partial_j^{r}p}{p}\right)(\textbf{a}).\]
\end{proposition}
\begin{proof}
Writing out the Taylor expansion of $p$, we have
\[p(\textbf{a}-\delta e_j) = \sum_{k = 0}^{r} \left(\partial_j^k p\right)(\textbf{a})\dfrac{(-1)^k\delta^k}{k!},\]
since $p$ is of degree $r$ in the variable $z_j$. Consequently, we have that
\[\left(\partial^{r-1}_j p\right)(\textbf{a}-\delta e_j) = \partial^{r-1}_j p(\textbf{a})
-\delta \partial^{r}_j p(\textbf{a}).\]
 We therefore seek the largest $\delta$ such that
\[ \dfrac{\partial_i \partial_j^{r-1} p - \delta \partial_i \partial^r_j p}{ \partial_j^{r-1} p - \delta \partial^r_j p} (\textbf{a}) \leq \dfrac{\partial_i p}{p}(\textbf{a}).\]
We may rewrite this as
\begin{eqnarray}\label{exp2}
\left[p(\partial_i \partial_j^{r-1} q) - (\partial_i p)(\partial_j^{r-1} p)\right](\textbf{a}) \leq \delta[p(\partial_i \partial_j^{r} p) - (\partial_i p)(\partial_j^{r} q) ](\textbf{a}).
\end{eqnarray}
This in turn can be written as
\[\partial_i\left(\dfrac{\partial_j^{r-1} p}{p}\right)(\textbf{a}) \leq \delta \,\partial_i\left(\dfrac{\partial_j^r p}{p}\right)(\textbf{a}),\]
which is what was claimed.
\end{proof}

The quantity above can be controlled geometrically.

\begin{proposition}\label{finest}
 Let $p(\textbf{z})  = p(z_1, \cdots, z_n)$ be a real stable polynomial of degree at most $r$ in $z_j$ and let $\textbf{a} = (a_1, \ldots, a_n) \in Ab_p$. Then,
 \[\partial_i\left(\dfrac{\partial_j^{r-1}p}{p}\right)(\textbf{a})  \leq \delta \,\partial_i\left(\dfrac{\partial_j^{r}p}{p}\right)(\textbf{a}),\]
provided
\[\delta \leq \dfrac{(r-1)^2}{r} \left(\dfrac{1}{\Phi_p^j(\textbf{a}) - \dfrac{1}{a_j-\lambda_r}}\right),\]
where $\lambda_r$ is the smallest root of the univariate polynomial $p(a_1, \ldots, a_{r-1}, z_j, a_{r},\ldots, a_n)$.
 
 \end{proposition}

\begin{proof}

We need to find a $\delta$ such that
\begin{align}\label{eqn0}
 \partial_i\left(\dfrac{\partial_j^{r-1}p}{p}\right) (\textbf{a}) \leq \delta \partial_i\left(\dfrac{\partial_j^{r}p}{p}\right)(\textbf{a}).
\end{align}

In what follows, we drop mentioning the reference point $\textbf{a}$ explicitly, to lighten the notation. We may expand out $p$ in the $z_j$ variable as a product,
\[p = g\,\prod_{k=1}^{r} \left(z_j-\lambda_k\right) = g \left[z_j^r - z_j^{r-1}\left(\sum_{k = 1}^r \lambda_k\right) + \cdots\right],\]
where $g$ and  $\lambda_1, \ldots, \lambda_{r}$ are functions of $\{z_k \mid k \in [n], k \neq j\}$. We also index the roots so that $\lambda_1 \geq \ldots \geq \lambda_r$. This yields
\[\left(\dfrac{\partial_j^{r-1}p}{p}\right) = \dfrac{ a_j r! - \left(\sum_{k = 1}^{r} \lambda_k\right) (r-1)!}{\prod_{k = 1}^r (a_j-\lambda_k)},\qquad  \left(\dfrac{\partial_j^{r}p}{p}\right) = \dfrac{r!}{\prod_{j=1}^r (a_j-\lambda_k)}.\]
Eqn. \ref{eqn0} becomes
\[a_j\partial_i  \left(\dfrac{1}{\prod_{k=1}^r (a_j-\lambda_k)}\right)- \dfrac{1}{r}\partial_i\left(\dfrac{\sum_{k = 1}^{r} \lambda_k}{\prod_{k=1}^r (a_j-\lambda_k)}\right) \leq \delta \partial_i \left(\dfrac{1}{\prod_{k=1}^r (a_j-\lambda_k)}\right).\]
This may be rewritten as
\begin{eqnarray*}
 -\left(\dfrac{1}{r}\sum_{k = 1}^{r} \partial_i \lambda_k\right) \dfrac{1}{\prod_{k=1}^r (a_j-\lambda_k)} &\leq& \left(\delta - a_j + \dfrac{\sum_{k=1}^r \lambda_k}{r}\right) \partial_i \left(\dfrac{1}{\prod_{k=1}^r 
 (a_j-\lambda_k)}\right),\\
 &=& \left(\delta - a_j +\dfrac{\sum_{k=1}^r \lambda_k}{r}\right) \left(\sum_{k = 1}^{r} \dfrac{\partial_i\lambda_k}{a_j-\lambda_l}\right) \dfrac{1}{\prod_{k=1}^r (a_j-\lambda_k)}.
\end{eqnarray*}
Since $\textbf{a}$ is above the roots of $p$, the product term $\prod_{k=1}^r (a_j-\lambda_k)$ is positive. Rearranging, we may write this as
\[\sum_{k = 1}^{r} \dfrac{-\partial_i \lambda_k}{a_j-\lambda_k}\,\left[\delta- \sum_{l \neq k} \dfrac{a_j-\lambda_l}{r}\right]\leq 0.\]
Note that each $\partial_j \lambda_k$ is non-positive by Proposition \ref{tao} and thus each of the left terms in the sum is positive. The derivative might not exist, but it does exist generically and we can make a small perturbation to ensure it does. Recall that we have $\lambda_1 \geq \ldots \geq \lambda_r$. The above inequality is satisfied term by term and hence in sum, provided
\begin{align}\label{eqn1}
\delta \leq \sum_{k = 1}^{r-1} \dfrac{a_j-\lambda_k}{r}
\end{align}
By the harmonic mean inequality,
\[\dfrac{(r-1)^2}{r} \left(\dfrac{1}{\Phi_p^j(\textbf{a}) - \dfrac{1}{a_j-\lambda_r}}\right) = \dfrac{(r-1)^2}{\sum_{k = 1}^{r-1} \dfrac{r}{a_j-\lambda_k}} \leq \sum_{k = 1}^{r-1} \dfrac{a_j-\lambda_k}{r}.\]
Therefore, the required inequality (\ref{eqn1}) is satisfied provided
\[\delta \leq \dfrac{(r-1)^2}{r} \left(\dfrac{1}{\Phi_p^j(\textbf{a}) - \dfrac{1}{a_j-\lambda_r}}\right) .\]
\end{proof}

To continue this analysis, we will need estimates of the following sort. 
\begin{question}
 Given a real stable polynomial $p(z_1, \cdots, z_n)$ and a point $(a_1, \cdots, a_n)$ above its roots, get (lower) bounds for the minimum root of the polynomial $p(z_1, a_2, \cdots, a_n)$. 
\end{question}

Bounds of this kind can be given for polynomials of the form $\operatorname{det}[Z-A]^r$ and their partial derivatives. We first prove a simple lemma.


\begin{lemma}\label{sing}
  Let $A$ be a PSD matrix, let $p(Z) = \operatorname{det}[Z-A]$ and let $\textbf{a}$ be above the roots of $p$. Then, the (single) root of $p(z_1, a_2, \ldots, a_n)$ is nonnegative.
\end{lemma}
\begin{proof}
    First, it is straightforward to see that $\textbf{a}$ is above the roots of $p$ iff $\operatorname{diag}(\textbf{a}) - A$ is positive definite. Consider the following linear polynomial:
    \[
        q(z_1) := p(z_1,a_2,...,a_n) = \operatorname{det}[\operatorname{diag}(z_1,a_2,...,a_n) - A]
    \]
    Since the eigenvalues of $Z-A$ are continuous with respect to $z_1$ and since $\operatorname{diag}(\textbf{a}) - A$ is positive definite, we have that $\operatorname{diag}(z_1,a_2,...,a_n) - A$ is positive definite iff $z_1$ is larger than the (single) root of $q$. Since positive definite matrices have positive diagonal entries, this means that the root of $q$ is at least the top left diagonal entry of $A$. The result follows.
\end{proof}



We remark that this fact can also be proved using a general fact that the set of points above the roots of a real stable polynomial is a convex set. However, we have given the above proof as it is short, concise and elementary. We bootstrap this result to cover the polynomials of interest to us.

\begin{proposition}
 Let $A \in M_n(\mathbb{C})$ be a PSD contraction such that all the diagonal entries are at least $\alpha$ and let $\textbf{a}$ be above the roots of $p(z_1, \ldots, z_n) = \operatorname{det}[Z-A]^r$. Then, for any multiset $(i_1,\ldots,i_n)$ where $0 \leq i_k \leq r-1$ for $k \in [n]$ and any $i \in [n]$, all the roots of the polynomial 
 \[q(z_i)= \left[\left(\prod_{k \in [n]} \partial_k^{i_k}\right)p\right]( \ldots, a_{i-1}, z_i, a_{i+1}, \ldots),\] are at least $\alpha$. In particular, they are all positive. 
\end{proposition}
\begin{proof}
By the Leibnitz formula, 
\[\left(\prod_{k \in [n]} \partial_k^{i_k}\right) \operatorname{det}[Z-A]^r = \sum_{ (T_1,\cdots T_r) \in \mathcal{S}} \prod_{k=1}^r \operatorname{det}[(Z-A)(T_k)],\]
for a suitable subset $S \subset [n]\times \cdots\times [n]$, the form of which will not be material to our proof.
By Lemma \ref{sing}, we see that each of the polynomials $\operatorname{det}[(Z-A)(T_i)]$ is positive at all points $(b_1, a_2, \ldots, a_n)$ where $b_1 \leq \alpha$. The same thus holds for $p$, proving the required result.
\end{proof}

In the next section, we combine these results to prove our improved estimates on the paving problem.

\section{The Multivariate barrier method : Continued}

The results in the previous section show how the barrier functions $\Phi_p^i$ of a real stable polynomial change upon iterated derivatives. To use this technique, we will need estimates on the barrier functions  of the function $\operatorname{det}[Z-A]^r$ , that we begin with. These turn out to be easy to calculate. 

 We will need the following well known result concerning the determinants of principal submatrices.
 
\begin{lemma}[Determinants of defect $1$ principal submatrices]\label{Com}
 For any matrix $A \in M_n(\mathbb{C})$ and any vector $v \in \mathbb{C}^n$, we have
\[\operatorname{det}\left(A_{v^{\perp}}\right) = \operatorname{det}(A) \left(v^{*}A^{-1}v\right),\]
where $A_{v^{\perp}} \in M_{n-1}(\mathbb{C})$ is the compression of $A$ onto $v^{\perp}$.
\end{lemma}

The barrier functions of powers of determinantal polynomials can be estimated as follows. 

\begin{lemma}\label{bes1}
Let $A \in M_n(\mathbb{C})$ be PSD and let $p(\textbf{z}) = \operatorname{det}[Z-A]^r$. Then
\[\Phi_p^i(\textbf{z}) := \dfrac{\partial_i p}{p}(\textbf{z}) = re_{i}^{*}(Z-A)^{-1}e_{i},\]
whenever $Z-A$ is invertible.
\end{lemma}
\begin{proof}
We have that, 
\[\Phi_p^i(\textbf{z}) = \dfrac{\partial_i p}{p}(\textbf{z})  = \dfrac{\partial_i \operatorname{det}[Z-A]^r}{\operatorname{det}[Z-A]^r} = \dfrac{r\partial_i \operatorname{det}[Z-A]}{\operatorname{det}[Z-A]}.\]
It is easy to see that
\[\partial_i \operatorname{det}[Z-A]  =  \operatorname{det}[\left(Z-A\right)_i] = \operatorname{det}[Z-A] e_{i}^{*}(Z-A)^{-1}e_{i}.\]
The first is an elementary calculation while the second follows from Lemma \ref{Com}. We conclude that
\[\Phi_p^i(\textbf{z}) = re_{i}^{*}(Z-A)^{-1}e_{i}.\]

\end{proof}

The quantity on the right can be controlled by the diagonal entries of the matrix $A$. We give here a first order estimate.

\begin{lemma}\label{barval}
Let $A \in M_n(\mathbb{C})$ be a PSD contraction and let the diagonal entries all be at most $\alpha$. Then, for any $a \geq 1$, 
\[e_{i}^{*}(aI-A)^{-1}e_{i} \leq \dfrac{\alpha}{a-1} + \dfrac{1-\alpha}{a}, \quad i \in [n].\]
\end{lemma}

\begin{proof}
Let $D$ be the diagonal matrix of eigenvalues of $A$. 
\[D  := \operatorname{diag}(\lambda_1, \cdots, \lambda_n),\]
 and let $U$ be a unitary matrix such that 
 \[A = UDU^{*}.\] We see that 
\[e_{i}^{*}(aI-A)^{-1}e_{i} = \left(U^{*}e_{i}\right)^{*}(aI-D)^{-1}\left(U^{*}e_{i}\right) = \sum_{j=1}^n \dfrac{|U_{ij}|^2}{a - \lambda_j}.\]
The condition that the diagonal entries are all at most $\alpha$ yields 
\begin{align}\label{diago} 
A_{ii} = \sum_{j=1}^n \lambda_j |U_{ij}|^2  \leq \alpha, \quad i \in [n].
\end{align}

Since $U$ is unitary, we also have that
\[\sum_{j=1}^n |U_{ij}|^2 =1 , \quad i \in [n].\]
The harmonic mean inequality shows that for any $\lambda \in [0,1]$ and $a > 1$, 
\[\dfrac{1}{a-\lambda} \leq \dfrac{\lambda}{a - 1} + \dfrac{1-\lambda}{a}.\]
Therefore,
\begin{align*}
\sum_{j=1}^n \dfrac{|U_{ij}|^2}{a - \lambda_j} \leq& \sum_{j=1}^n \left(\dfrac{\lambda_j |U_{ij}|^2}{a - 1} + \dfrac{(1-\lambda_j) |U_{ij}|^2}{a}\right),\\
=& \dfrac{\sum_{j=1}^n \lambda_j |U_{ij}|^2}{a-1} +  \dfrac{\sum_{j=1}^n (1-\lambda_j ) |U_{ij}|^2}{a},\\
=& \left(\sum_{j=1}^n \lambda_{j} |U_{ij}|^2\right)\left(\dfrac{1}{a-1}-\dfrac{1}{a}\right) +  \dfrac{1}{a},\\
\leq & \dfrac{\alpha}{a-1} + \dfrac{1-\alpha}{a}.\tag{By Ineq. \ref{diago}}
\end{align*}
\end{proof}

Combining Lemmas \ref{bes1} and \ref{barval}, we conclude,
\begin{lemma}\label{barest}
Let $A \in M_n(\mathbb{C})$ be a PSD contraction with diagonal entries all at most $\alpha$ and let $p = \operatorname{det}[Z-A]^r$. Then, for any $a \geq 1$, 
\[\Phi_p^i(a\textbf{1}) \leq   r\left(\dfrac{\alpha}{a-1} + \dfrac{1-\alpha}{a}\right).\]
\end{lemma}

We will need the following simple optimization result in the proof of the main theorem.

\begin{lemma}\label{opt}
Let $\alpha, \beta$ be real numbers in $[0,1]$. Then,
\[\operatorname{inf}_{a > 1}\, a - \dfrac{\beta}{\dfrac{\alpha}{a-1}+\dfrac{1-\alpha}{a}} =\begin{cases} \left(\sqrt{\alpha \beta} + \sqrt{(1-\alpha)(1-\beta)}\right)^2 \leq 1, \quad  \alpha \leq \beta,\\
1, \quad \alpha \geq \beta
\end{cases}.\]
\end{lemma}

We are now ready to prove our main theorem. 

\begin{theorem}\label{Main}
Let $A \in M_n(\mathbb{C})$ be a PSD contraction with diagonal entries all at most $\alpha$. Then, for any integral $r \geq 2$, such that 
\[\dfrac{(r-1)^2}{r^2} \geq \alpha,\]
we have that
\[\operatorname{max root}\chi_r[A]\leq \left(\sqrt{\dfrac{1}{r}-\dfrac{\alpha}{r-1}}+ \sqrt{\alpha}\right)^2.\]
\end{theorem}

\begin{proof}
Let $p$ be the polynomial,
\[p(\textbf{z}) = \operatorname{det}\left[Z-A\right]^r,\]
 and let $a > 1$ and let $\textbf{b}_0 = a\textbf{1} \in \mathbb{R}^n$. Since $A$ is a PSD contraction, the vector $\textbf{b}_0$ is above the roots of $p$.
 Let us iteratively define the polynomials
 \[p_1 = \partial_1^{r-1}p, \quad p_k = \partial_k^{r-1}p_{k-1}, \quad  k = 2, \ldots,n.\]

 Also iteratively define for $k = 1, \ldots, n$, the shift $\delta_k$ and the vector $\textbf{b}_k$ by 
 \[\delta_k =  \dfrac{(r-1)^2}{r} \left(\dfrac{1}{\Phi_{p_{k-1}}^k(\textbf{b}_{k-1}) - \dfrac{1}{a}}\right), \quad \textbf{b}_k = \textbf{b}-\sum_{i = 1}^{k} \delta_i e_i.\]

Combining Prop. \ref{finest} and Lem. \ref{sing}, we see that 
 \begin{align}\label{est}
 \Phi_{p_k}^{i}(\textbf{b}_k) \leq \Phi_{p_{k-1}}^i(\textbf{b}_{k-1}), \quad k \in [n], \, i \in [n].
 \end{align}

Consequently, for $k = 1,\ldots,n$
\begin{align*}
\delta_k =  \dfrac{(r-1)^2}{r} \left(\dfrac{1}{\Phi_{p_{k-1}}^k(\textbf{b}_{k-1}) - \dfrac{1}{a}}\right) &\geq  \dfrac{(r-1)^2}{r} \left(\dfrac{1}{\Phi_{p}^k(\textbf{b}_{0}) - \dfrac{1}{a}}\right),\\
&=\dfrac{(r-1)^2}{r} \left(\dfrac{1}{\Phi_{p}^k(a\textbf{1}) - \dfrac{1}{a}}\right).
\end{align*}

This in turn implies that the vector
\[b = (a-v)\textbf{1}, \quad v := \dfrac{(r-1)^2}{r} \dfrac{1}{\Phi_p^i(a\textbf{1}) - \dfrac{1}{a}}, \]
is above the roots of $(\partial_1 \cdots \partial_n)^{r-1}\operatorname{det}[Z-A]^r$ and thus, 
\[\min_{a \geq 1} \left\{a-\dfrac{(r-1)^2}{r} \left(\dfrac{1}{\Phi_p^i(a\textbf{1}) - \dfrac{1}{a}}\right)\right\},\]
is larger than the largest root of $\chi_r[A]$.
Using the fact, see Lem. \ref{barest}, that
\[\Phi_p^i(a\textbf{1}) \leq r\left(\dfrac{\alpha}{a-1}+\dfrac{1-\alpha}{a}\right),\]
we see that
\begin{align*}
\operatorname{max root}\, \chi_r[A] \leq& \min_{a \geq 1} \left\{a - \dfrac{(r-1)^2}{r} \left(\dfrac{1}{r\left(\dfrac{\alpha}{a-1}+\dfrac{1-\alpha}{a}\right) - \dfrac{1}{a}}\right)\right\},\\
=&\min_{a \geq 1} \left\{a - \dfrac{(r-1)}{r} \left(\dfrac{1}{\dfrac{r\alpha/(r-1)}{a-1}+\dfrac{1-r\alpha/(r-1)}{a}}\right)\right\}.
\end{align*}

Using Lem. \ref{opt}, we see that when $(r-1)^2/r^2 > \alpha$, 
\[\operatorname{max root}\chi_r[A]\leq \left(\sqrt{\dfrac{1}{r}-\dfrac{\alpha}{r-1}}+ \sqrt{\alpha}\right)^2 \leq 1,\]
\end{proof}

Applying this result with $\alpha  = 1/2$ and $r = 4$ yields Cor. \ref{dhalf} that says that PSD contractions with diagonal at most $1/2$ can be $4$ paved.  

\section{Concluding remarks}

We briefly place the calculations in this paper in the context of polynomial convolutions, something that clarifies the issues related to obtaining optimal estimates in the Paving problem.

In \cite{FFree}, Marcus, Spielman and Srivastava discuss a convolution on polynomials that they call the Symmetric Additive Convolution and prove root bounds for this operation.

\begin{definition}[MSS]
 Let $A, B \in M_n(\mathbb{C})$ be Hermitian matrices and let $p(z) = \operatorname{det}[zI-A]$ and $q(z) = \operatorname{det}[zI-B]$ be their characteristic polynmials. The symmetric additive convolution is defined as
 \[\left(p \boxplus_n q\right)(z) := \mathbb{E}_{O \in \mathcal{O}(n)}\operatorname{det}\left(zI-A-OBO^{T}\right).\]
\end{definition}

MSS showed that this is identical to a convolution on polynomials introduced and studied by Walsh in 1905 and which had been noted by him to preserve real rootedness. MSS gave root bounds for this convolution as follows: Recall the barrier function  $\Phi_p = p'/p$ for any polynomial $p$. MSS showed that if  $\Phi_p(a) \leq \varphi$ and $ \Phi_q(b) \leq \varphi$ where $a, b$ are larger than the max roots of $p$ and $q$ respectively, then
\[\Phi_{p \boxplus_n q}\left(a + b - \dfrac{1}{\varphi}\right) \leq \varphi.\]

In the setting of the paving problem, the expected characteristic polynomials can be cast in this framework. Let $A \in M_n(\mathbb{C})$. The set of all $2$ pavings of $A$ is precisely the set $(A + DAD)/2$ where $D$ ranges over the set of diagonal matrices with each diagonal entry in $\{-1,1\}$. As a consequence, we have have that 
\[\chi_2[A](z) = \mathbb{E}_{D \in \mathcal{D}(n)}\operatorname{det}\left(2zI-A-DAD\right),\]
where $\mathcal{D}_n$ is the set of diagonal matrices with diagonal entries all in $\{\pm 1\}$. This prompts the natural definition,
\begin{definition}
 Let $A, B \in M_n(\mathbb{C})$ be Hermitian matrices and let $p(\textbf{z}) = \operatorname{det}[Z-A]$ and $q(\textbf{z}) = \operatorname{det}[Z-B]$ be naturally affiliated multiaffine polynomials. The (multivariate) symmetric additive convolution is defined as
 \[\left(p \boxplus_n q\right)(\textbf{z}) := \mathbb{E}_{D \in \mathcal{D}(n)}\operatorname{det}\left(Z-A-DBD\right).\]
\end{definition}

This convolution can be interpreted as taking as input two multiaffine real stable polynomials and returning another multiaffine real stable polynomial. It is now natural to ask if the natural generalization of MSS' root shift bound holds in this multivariate setting as well.

\begin{question}
 Let $A, B \in M_n(\mathbb{C})$ be Hermitian matrices and let $p(\textbf{z}) = \operatorname{det}[Z-A]$ and $q(\textbf{z}) = \operatorname{det}[Z-B]$ and let $\textbf{a}$ and $\textbf{b}$ be above the roots of $p$ and $q$ respectively. Suppose we have that 
 \[\Phi_p^i(\textbf{a}) \leq \varphi_i, \qquad \Phi_q^i(\textbf{b}) \leq \varphi_i, \quad i \in [n],\]
 for some positive constants $\varphi_1, \ldots, \varphi_n$. Then is it true that
 \[\Phi_{p \boxplus_n q}^i\left(\textbf{a}+\textbf{b}-\dfrac{1}{\varphi_i}\right) \leq \varphi_i, \quad i \in [n]?\]
\end{question}

If this were true, one can show that this would yield optimal estimates in the paving problem, see Conj. \ref{conj}.

There is also a natural way to generalize this convolution to real stable polynomials in general, but we restrict our attention to the multiaffine case in this discussion. The answer to the corresponding root bounds question is unfortunately ``no'', failing even for polynomials in 3 variables. This fact is due to Leake and Ryder, and an explicit counterexample will be given in forthcoming work. It is currently unclear if counterexamples for polynomials of the form $p(\textbf{z}) = \operatorname{det}[Z-A]$ exist.

That said, the symmetric additive convolution has been one of the main approaches to generalizing the results of MSS to more general classes of polynomials. The fact that the root bound they achieve breaks down for multivariate real stable polynomials in general suggests one of two things: either the additive convolution is not the correct object of study, or we must restrict our attention to specific types of real stable polynomials.

The $r$ characteristic polynomial we discuss here then becomes a guide for this line of thought. If one is hoping to obtain optimal paving bounds, then constructing a more general theory should be oriented around what works for $\chi_r[A]$. We hope that this paper can be a first step towards understanding the $r$ characteristic polynomial, as well as its place within the broader theory of root bounds on real stable polynomials.

\section{Acknowledgements}
This paper derives from several conversations the second author had with Betul Tanbay who proposed trying to understand if the machinery of MSS could be directly applied to Anderson's paving conjecture. He'd like to thank her for all her time and feedback. He'd also like to thank Ozgur Martin for discussions and his insights. This research was supported by TUBITAK 1001 grant number 115F204, ``Geometric questions in von Neumann algebras''. The first author would like to thank the Institute Mittag-Leffler where parts of this work were done.

\end{document}